\documentclass{article}
\usepackage[utf8]{inputenc}
\usepackage{amsmath,amsfonts,amsthm,amssymb}
\usepackage{microtype}
\usepackage[margin=1.5in]{geometry}
\newtheorem{lemma}{Lemma}
\newtheorem{theorem}{Theorem}
\newtheorem{corollary}{Corollary}

\title{A Correspondence Between Maximal Surfaces and Timelike Minimal Surfaces in $\mathbb{L}^3$}
\author{Aryaman Patel }

\begin{document}
\date{}
\maketitle

\begin{abstract}
    We show that to every maximal surface with conelike singularities in Lorentz-Minkowski space $\mathbb{L}^3$ that can be locally represented as the graph of a smooth function, there exists a corresponding timelike minimal surface in $\mathbb{L}^3$. There exists a linear transformation between such a maximal surface and its corresponding timelike minimal surface and it maps the singularities of one to the singularities of the other. Moreover, this transformation establishes a one-one correspondence between such maximal surfaces and timelike minimal surfaces and also preserves the one-one property of the Gauss map. This leads to a Kobayashi type theorem for timelike minimal surfaces in $\mathbb{L}^3$. Finally, we derive some non-trivial identities using existing Euler-Ramanujan identities, and some familiar timelike minimal surfaces in parametric form.  
\end{abstract}

\section{Introduction}

Maximal surfaces and timelike minimal surfaces in Lorentz-Minkowski space $\mathbb{L}^3$ have been studied extensively in the recent years and they are well characterized\cite{Koba,Koba2}. In this article we further explore the relationship between them.\\
We consider maximal surfaces and timelike minimal surfaces with conelike singularities that can be locally expressed as graphs of smooth functions. Any graph $\varphi=(x,y,z(x,y))$ is a maximal surface in $\mathbb{L}^3$ if the function $z(x,y)$ is smooth and satisfies the following partial differential equation
\begin{align*}
    (1-z_x^2)z_{yy}+2z_xz_yz_{xy}+(1-z_y^2)z_{xx}=0
\end{align*}
known as the maximal surface equation. Similarly, any graph $\psi=(u,v,w(u,v))$ is a timelike minimal surface in $\mathbb{L}^3$ if it satisfies the following partial differential equation
\begin{align*}
    (w_u^2-1)w_{vv}-2w_uw_vw_{uv}+(1+w_v^2)w_{uu}=0
\end{align*}
known as the Born-Infeld equation.\\
It has been shown in \cite{DS} that the Born-Infeld equation and the maximal surface equation are related by a Wick rotation in the first variable. The Wick rotation is a well-known linear transformation which gives a way to generate a maximal surface from a timelike minimal surface, and vice-versa. This naturally leads to the question of whether one can associate to any maximal surface in $\mathbb{L}^3$ a corresponding timelike minimal surface via a linear transformation. We show that this is true for certain maximal surfaces, and that the transformation is indeed linear with some nice properties. This helps characterize timelike minimal surfaces with conelike singularities whose Gauss map is one-one, using Kobayashi's result\cite{Koba2}. In the spirit of the second part of \cite{DS}, we use Ramanujan's identities and parametric forms of certain timelike minimal surfaces to derive non-trivial identities.

\section{ Timelike minimal surfaces with conical singularities}

\begin{lemma}
For every maximal surface in $\mathbb{L}^3$, there exists a corresponding timelike minimal surface and the correspondence is given by a \emph{linear transformation}.
\end{lemma}

\begin{proof}
Let $\varphi$ be a maximal surface in $\mathbb{L}^3$. Locally, $\varphi$ can be expressed as the graph of some smooth function $z(x,y)$ of canonical variables $x,y$ i.e. $\varphi(x,y)=(x,y,z(x,y))$. Then, the function satisfies the following partial differential equation 
\begin{equation}
    (1-z_x^2)z_{yy}+2z_xz_yz_{xy}+(1-z_y^2)z_{xx}=0,
\end{equation}
also known as the maximal surface equation. Now consider a general transformation $ \psi: (x,y)\to (u,v)$ such that the resultant surface  is a timelike minimal surface in $\mathbb{L}^3$ of the form $(u, v, z(u,v))$ such that $z(u,v)$ is a smooth function that satisfies the following partial differential equation
\begin{equation}
    (z_u^2-1)z_{vv}-2z_uz_vz_{uv}+(1+z_v^2)z_{uu}=0.
\end{equation}

 In fact the most general local transformation (preserving regularity) from a maximal surface to  its corresponding  timelike minimal surface can be realised like this.   First we make the transformation  $\phi(x,y) = (u_1, v_1)$ in which the corresponding  timelike minimal surface can be written as $(x(u_1, v_1), y( u_1, v_1), z(u_1, v_1))$ and then (possibly rotating the timelike minimal surface) we make  a further change of variables $ (u,v) =\tau(u_1, v_1)$  such that the timelike  minimal surface is of the form $ (u, v, z(u,v))$, i.e. it is the graph of a function. Denote $\psi = \tau \circ \phi$.  

We shall show that $\psi $ is a linear transformation.

Using the fact that $\frac{\partial}{\partial u}=\frac{\partial}{\partial x}\frac{\partial x}{\partial u}+\frac{\partial}{\partial y}\frac{\partial y}{\partial u}$ and $\frac{\partial}{\partial v}=\frac{\partial}{\partial x}\frac{\partial x}{\partial v}+\frac{\partial}{\partial y}\frac{\partial y}{\partial v}$, we can rewrite the partial derivatives $z_u$ and $z_v$ as follows
\begin{align*}
    z_u=x_uz_x+y_uz_y,\,\,\,z_v=x_vz_y+y_vz_y.
\end{align*}
Similarly, the second partial derivatives $z_{uu}$, $z_{uv}$ and $z_{vv}$ can be rewritten as follows
\begin{align*}
    z_{uu}=x_{uu}z_x+x_u^2z_{xx}+2x_uy_uz_{xy}+y_{uu}z_y+y_u^2z_{yy}\\
    z_{uv}=x_{uv}z_x+x_ux_vz_{xx}+(x_uy_v+x_vy_u)z_{xy}+y_{uv}z_y+y_uy_vz_{yy}\\
    z_{vv}=x_{vv}z_x+x_v^2z_{xx}+2x_vy_vz_{xy}+y_{vv}z_y+y_v^2z_{yy}.
\end{align*}
Using the above expressions, the timelike minimal surface equation (2) can be rewritten as follows
\begin{align*}
    (x_u^2z_x^2+y_u^2z_y^2+2x_uy_uz_xz_y-1)(x_{vv}z_x+x_v^2z_{xx}+2x_vy_vz_{xy}+y_{vv}z_y+\\y_v^2z_{yy})
    -2(x_ux_vz_x^2+y_uy_vz_y^2+(x_uy_v+y_ux_v)z_xz_y)(x_{uv}z_x+x_ux_vz_{xx}\\+(x_uy_v+x_vy_u)z_{xy}+y_{uv}z_y+y_uy_vz_{yy})
    +(x_v^2z_x^2+y_v^2z_y^2+2x_vy_vz_xz_y\\+1)(x_{uu}z_x+x_u^2z_{xx}+2x_uy_uz_{xy}+y_{uu}z_y+y_u^2z_{yy})=0.
\end{align*}

Since it is known that $z(x,y)$ satisfies the maximal surface equation, we require that the above equation is  exactly  the equation (1). This leads to a system of twelve equations in the first and second partial derivatives of $x$ and $y$ with respect to $u$ and $v$. We split this system into two sets as follows. The first set is a system of four equations in only the first partial derivatives of $x$, $y$ with respect to $u$, $v$ given by 
\begin{align}
    x_u^2-x_v^2=1,\,\,\,y_u^2-y_v^2=1,\,\,\,(x_uy_v-x_vy_u)^2=-1,\,\,\,x_uy_u-x_vy_v=0.
\end{align}
The second set is a system of eight equations in the first and second partial derivatives of $x$, $y$ with respect to $u$, $v$ given by 
\begin{align*}
    x_{vv}x_u^2-2x_{uv}x_ux_v+x_{uu}x_v^2=0,\,\,\,x_{vv}y_u^2-2x_{uv}y_uy_v+x_{uu}y_v^2=0,\\
    y_{vv}x_u^2-2y_{uv}x_ux_v+y_{uu}x_v^2=0,\,\,\,y_{vv}y_u^2-2y_{uv}y_uy_v+y_{uu}y_v^2=0,\\
    x_{vv}x_uy_u+x_{uu}x_vy_v-x_{uv}(x_uy_v+x_vy_u)=0,\\
    y_{vv}x_uy_u+y_{uu}x_vy_v-y_{uv}(x_uy_v+x_vy_u)=0,\\
    x_{uu}-x_{vv}=0,\,\,\,y_{uu}-y_{vv}=0.
\end{align*}
From the last line in the above system of equations, we can set $x_{uu}=x_{vv}=a$, $x_{uv}=b$ and $y_{uu}=y_{vv}=c$, $y_{uv}=d$. Using any one of the two equations in the first line and the equation in the third line of the above expression, we obtain a system of two equations in $a$ and $b$. For example,
\begin{align*}
    a(x_u^2+v_v^2)-2bx_ux_v=0,\\
    a(x_uy_u+x_vy_v)-b(x_uy_v+x_vy_u)=0.
\end{align*}
Using the equations from those in system (3), we immediately get $a=0$, $b=0$ as the solution to the above simultaneous equations. Similarly, we obtain $c=0$, $d=0$, which implies that all double partial derivatives of $x$, $y$ with respect to $u$, $v$ are zero. Thus it follows that the transformation $(x,y)\to (u,v)$ which takes the maximal surface $\varphi(x,y)=(x,y,z(x,y))$ to the timelike minimal surface $\psi(u,v)=(u,v,z(u,v))$ must be a linear transformation.

\end{proof}

\begin{corollary}
The transformation that maps the maximal surface in $\mathbb{L}^3$ represented locally as $\varphi(x,y)=(x,y,z(x,y))$ to the timelike minimal surface in $\mathbb{L}^3$ represented locally as $\psi(u,v)=(u,v,z(u,v))$ also maps the singular points of $\varphi$ to the singular points of $\psi$.
\end{corollary}

\begin{proof}
The singular points of a maximal or timelike minimal surface in $\mathbb{L}^3$ are exactly those points at which the metric degenerates. For the maximal surface represented locally as $\varphi(x,y)=(x,y,z(x,y))$, the singular points are exactly those points $(x,y)$ where 
\begin{equation}
    z_x^2+z_y^2-1=0.
\end{equation}
Similarly, for the timelike minimal surface represented locally as $\psi(u,v)=(u,v,z(u,v))$, the singular points are exactly those points $(u,v)$ where
\begin{equation}
    z_u^2-z_v^2-1=0.
\end{equation}
Recall that the maximal surface $\varphi$ corresponds to the timelike minimal surface $\psi$ via the linear transformation $(x,y)\to (u,v)$. Thus, the partial derivatives $z_u$ and $z_v$ can be expressed as $z_u=x_uz_x+y_uz_y$ and $z_v=x_vz_x+y_vz_y$. The equation (5) can thus be rewritten as follows
\begin{align*}
    z_x^2(x_u^2-x_v^2)+z_y^2(y_u^2-y_v^2)+2z_xz_y(x_uy_u-x_vy_v)-1=0.
\end{align*}
From the system of equations (3), the above equation simplifies to $z_x^2+z_y^2-1=0$, but this is exactly the equation (4).\\
This concludes the proof.
\end{proof}
 
\begin{lemma}
The transformation from Lemma 1 establishes a one-one correspondence between maximal surfaces and timelike surfaces with conical singularities that can be locally represented as graphs. 
\end{lemma}

\begin{proof}
We show that by applying the transformation to a maximal surface, the corresponding timelike minimal surface we arrive at is unique.\\
Consider a maximal surface $\varphi=(x,y,z(x,y))$. Suppose $\psi=(u,v,z(u,v))$ and $\psi^{'}=(u^{'},v^{'},z(u^{'},v{'}))$ are two timelike minimal surfaces generated from $\varphi$ by applying linear transformations $\Phi$ and $\Phi^{'}$ respectively. Consider a transformation $\sigma$, which maps $\psi$ to $\psi^{'}$ and makes the triangle commute i.e., $\sigma\circ\Phi=\Phi^{'}$. Now, we may treat the coordinates $u^{'}$ and $v^{'}$ as functions of $u$ and $v$. We have as before $\frac{\partial}{\partial u^{'}}=\frac{\partial}{\partial u}\frac{\partial u}{\partial u^{'}}+\frac{\partial}{\partial v}\frac{\partial v}{\partial v^{'}}$ and $\frac{\partial}{\partial v^{'}}=\frac{\partial}{\partial u}\frac{\partial u}{\partial v^{'}}+\frac{\partial}{\partial v}\frac{\partial v}{\partial v^{'}}$ from which we arrive at 
\begin{align*}
    z_{u^{'}}=u_{u^{'}}z_u+v_{u^{'}}z_v,\;\;\;z_{v^{'}}=u_{v^{'}}z_u+v_{v^{'}}z_v.
\end{align*}
Similarly, the second partial derivatives $z_{u^{'}u^{'}}$, $z_{u^{'}v^{'}}$, $z_{v^{'}v^{'}}$ can be expressed as follows
\begin{align*}
    z_{u^{'}u^{'}}=u_{u^{'}u^{'}}z_u+u_{u^{'}}^2z_{uu}+2u_{u^{'}}v_{u^{'}}z_{uv}+v_{u^{'}u^{'}}z_v+v_{u^{'}}^2z_{vv}\\
    z_{u^{'}v^{'}}=u_{u^{'}v^{'}}z_u+u_{u^{'}}u_{v^{'}}z_{uu}+(u_{u^{'}}v_{v^{'}}+u_{v^{'}}v_{u^{'}})z_{uv}+v_{u^{'}v^{'}}z_v+v_{u^{'}}v_{v^{'}}z_{vv}\\
    z_{v^{'}v^{'}}=u_{v^{'}v^{'}}z_u+u_{v^{'}}^2z_{uu}+2u_{v^{'}}v_{v^{'}}z_{uv}+v_{v^{'}v^{'}}z_v+v_{v^{'}}^2z_{vv}.
\end{align*}
Now $z(u^{'},v^{'})$ satisfies the partial differential equation $(z_{u^{'}}^2-1)z_{v^{'}v^{'}}-2z_{u^{'}}z_{v^{'}}z_{u^{'}v^{'}}+(1+z_{v^{'}}^2)z_{u^{'}u^{'}}=0$. Substituting the above expressions for partial derivatives of $z$ with respect to $u^{'}$ and $v^{'}$ into this partial differential equation, we arrive at the following equation
\begin{align*}
    (u_{u^{'}}^2z_u^2+v_{u^{'}}^2z_v^2+2u_{u^{'}}v_{u^{'}}z_uz_v-1)(u_{v^{'}v^{'}}z_u+u_{v^{'}}^2z_{uu}+2u_{v^{'}}v_{v^{'}}z_{uv}+v_{v^{'}v^{'}}z_v+\\v_{v^{'}}^2z_{vv})
    -2(u_{u^{'}}u_{v^{'}}z_u^2+v_{u^{'}}v_{v^{'}}z_v^2+(u_{u^{'}}v_{v^{'}}+v_{u^{'}}u_{v^{'}})z_uz_v)(u_{u^{'}v^{'}}z_u+u_{u^{'}}u_{v^{'}}z_{uu}\\+(u_{u^{'}}v_{v^{'}}+u_{v^{'}}v_{u^{'}})z_{uv}+v_{u^{'}v^{'}}z_v+v_{u^{'}}v_{v^{'}}z_{vv})
    +(u_{v^{'}}^2z_u^2+v_{v^{'}}^2z_v^2+2u_{v^{'}}v_{v^{'}}z_uz_v\\+1)(u_{u^{'}u^{'}}z_u+u_{u^{'}}^2z_{uu}+2u_{u^{'}}v_{u^{'}}z_{uv}+v_{u^{'}u^{'}}z_v+v_{u^{'}}^2z_{vv})=0.
\end{align*}
Since $(u,v,z(u,v))$ is a timelike minimal surface by assumption, we require that the above equation is, in fact, the partial differential equation $(z_u^2-1)z_{vv}-2z_uz_vz_{uv}+(1+z_v^2)z_{uu}=0$. The calculation that follows is similar to that in the proof of Lemma 1. Comparing coefficients of partial derivatives of $z$ with respect to $u$ and $v$ in both equations, it is easy to see that all second partial derivatives $u_{u^{'}u^{'}}$, $u_{u^{'}v^{'}}$, $u_{v^{'}v^{'}}$, $v_{u^{'}u^{'}}$, $v_{u^{'}v^{'}}$, $v_{v^{'}v^{'}}$ are zero, which implies that $\sigma$ is a linear transformation. Further, we arrive at the equations $(u_{u^{'}}v_{v^{'}}+u_{v^{'}}v_{u^{'}})^2=1$, $v_{v^{'}}^2-v_{u^{'}}^2=1$, $v_{u^{'}}v_{v^{'}}=0$, $u_{u^{'}}^2-u_{v^{'}}^2=1$, and $u_{u^{'}}u_{v^{'}}=0$. Thus the determinant of the Jacobian matrix of the transformation $\begin{pmatrix}u_{u^{'}} & u_{v^{'}} \\v_{u^{'}} & v_{v^{'}}\end{pmatrix}$ is either 1 or -1. Since $\sigma$ is a linear transformation, this means that $\psi=(u,v,z(u,v))$ is the unique timelike minimal surface corresponding to $\varphi=(x,y,z(x,y))$ up to a rigid motion. \\
The fact that $\psi=(u,v,z(u,v))$ and $\psi^{'}=(u^{'},v^{'},z(u^{'},v^{'}))$ correspond to the same maximal surface $\varphi=(x,y,z(x,y))$ has been used, since they have the "same" third coordinate $z$ as $\varphi$.\\
Similarly, if $\varphi=(x,y,z(x,y))$ and $\varphi^{'}=(x^{'},y^{'},z(x^{'},y^{'}))$ are two maximal surfaces that map to the same timelike minimal surface $\psi=(u,v,z(u,v))$ via linear transformations $\Phi$ and $\Phi^{'}$, then it can be shown that $\varphi$ and $\varphi^{'}$ are the same maximal surface up to a rigid motion, in the same way as above. \\
We conclude that the linear transformation from Lemma 1 establishes a one-one correspondence between maximal surfaces and timelike minimal surfaces with conical singularities that can be locally represented as graphs in $\mathbb{L}^3$. 
\end{proof}

In order to arrive at a Kobayashi type result for timelike minimal surfaces, we require the map that transforms a maximal surface to its corresponding timelike minimal surface to preserve the one-one property of the Gauss map. We state this result as the following lemma.

\begin{lemma}
If the Gausss map of a maximal surface is one-one, then the Gauss map of the corresponding timelike minimal surface is also one-one. 
\end{lemma}

\begin{proof}
Consider a maximal surface $\varphi=(x,y,z(x,y))$ and let $\psi=(u,v,z(u,v))$ be the corresponding timelike minimal surface. The Gauss map maps a point $(x,y,z)$ on $\varphi$ to the vector $(z_x,z_y,1)$ which can be viewed as a point on the unit two-sheeted hyperboloid $X^2+Y^2-Z^2=-1$. Similarly, a point $(u,v,z(u,v))$ on $\psi$ is mapped to a point $(z_u,-z_v,1)$ on the unit one-sheeted hyperboloid $-U^2+V^2+W^2=1$ by the Gauss map.\\
We assume that the Gauss map of the maximal surface $\varphi$ is one-one. Now let $(x_1,y_1,z(x_1,y_1))$ and $(x_2,y_2,z(x_2,y_2))$ be two points $\varphi$ which are mapped to the points $(u_1,v_1,z(u_1,v_1))$ and $(u_2,v_2,z(u_2,v_2))$ respectively on the corresponding timelike minimal surface $\psi$ by the map $\Phi$. We denote, for convenience, the partial derivative $z_u$ evaluated at $(u_1,v_1)$ by $z_{u_1}$ i.e., $z_u(u_1,v_1)=z_{u_1}$. With this notation, the image of the Gauss map at $(u_1,v_1,z(u_1,v_1))$ is the point $(z_{u_1},-z_{v_1},1)$ and at $(u_2,v_2,z(u_2,v_2))$ is the point $(z_{u_2},-z_{v_2},1)$. Similarly, the image of the Gauss map at $(x_1,y_1,z(x_1,y_1))$ is the point $(z_{x_1},z_{y_1},1)$ and at $(x_2,y_2,z(x_2,y_2))$ is the point $(z_{x_2},z_{y_2},1)$. Using the expressions $\frac{\partial}{\partial x}=\frac{\partial}{\partial u}\frac{\partial u}{\partial x}+\frac{\partial}{\partial v}\frac{\partial v}{\partial x}$ and $\frac{\partial}{\partial y}=\frac{\partial}{\partial u}\frac{\partial u}{\partial y}+\frac{\partial}{\partial v}\frac{\partial v}{\partial y}$ we can write 
\begin{align*}
    (z_{x_1},z_{y_1},1)=(u_xz_{u_1}+v_xz_{v_1},u_yz_{u_1}+v_yz_{v_1},1)\\
    (z_{x_2},z_{y_2},1)=(u_xz_{u_2}+v_xz_{v_2},u_yz_{u_2}+v_yz_{v_2},1).
\end{align*}
Recall from Lemma 1 that all second partial derivatives of $x$ and $y$ with respect to $u$ and $v$ are zero, i.e. all first partial derivatives are constants, hence the transformation $\Phi$ is linear. The same is true of the inverse transformation and so the partial derivatives $u_x$, $u_y$, $v_x$ and $v_y$ may be viewed as constants.\\
Suppose that the images of the Gauss map at $(u_1,v_1,z(u_1,v_1))$ and $(u_2,v_2,z(u_2,v_2))$ agree, i.e., $(z_{u_1},-z_{v_1},1)=(z_{u_2},-z_{v_2},1)$ which implies $z_{u_1}=z_{u_2}$ and $z_{v_1}=z_{v_2}$. This further implies $(u_xz_{u_1}+v_xz_{v_1},u_yz_{u_1}+v_yz_{v_1},1)=(u_xz_{u_2}+v_xz_{v_2},u_yz_{u_2}+v_yz_{v_2},1)$, i.e., $(z_{x_1},z_{y_1},1)=(z_{x_2},z_{y_2},1)$ which means that the images of the Gauss map at $(x_1,y_1,z(x_1,y_1))$ and $(x_2,y_2,z(x_2,y_2))$ agree. Since we have assumed the Gauss map of $\varphi$ to be one-one, it must be that $(x_1,y_1,z(x_1,y_1))=(x_2,y_2,z(x_2,y_2))$ and it immediately follows that $(u_1,v_1,z(u_1,v_1))=(u_2,v_2,z(u_2,v_2))$.\\
To summarize, we have shown that if $(z_{u_1},-z_{v_1},1)=(z_{u_2},-z_{v_2},1)$ then $(u_1,v_1,z(u_1,v_1))=(u_2,v_2,z(u_2,v_2))$ i.e., that the Gauss map of $\psi$ is one-one given that the Gauss map of the corresponding maximal surface $\varphi$ is one-one. This concludes the proof. 
\end{proof}

\begin{theorem}
Let $S$ be a complete timelike minimal surface in $\mathbb{L}^3$ with at least one conelike singularity. Suppose the Gauss map of $S$ is one-one. Then $S$ is congruent to the surface defined by $\sqrt{x^2-y^2}+a\textrm{sinh}(\frac{z}{a})=0$, where $a$ is a nonzero real constant. 
\end{theorem} 

Note that the result is not true if the assumption that the Gauss map is one-one is not made. An example will be given in the following section.

\begin{proof}

We apply inverse of the correspondence, namely  $\Phi^{-1}$ which maps our timelike minimal surface $S$ to $\tilde{S}$ preserving the conical singularity.

Now, recall that $\Phi^{-1}$ involved a map linear map $\psi^{-1}: (x,y) \to (u,v)$.
If we extend $\psi^{-1}: \mathbb{C}^2 \to \mathbb{C}^2$ and let $(x,y, z) \in \mathbb{C}^3 $ and $(u,v, z) \in  \mathbb{C}^3$ the functional equations of maximal and timelike minimal equations do not change.
Thus $\psi^{-1}$ is a linear transformation from $\mathbb{C}^2 \to \mathbb{C}^2$ (by same argument as in lemma...).
Now making the linear transformation,  $y \to iy$, we get a maximal surface (in $\mathbb{ C}^3$).

Restricting to $(x,y) \in \mathbb{R}^2$, this gives a  unique real maximal surface (since it is given by the restriction of $\Phi^{-1}$). This surface is complete, has Gauss map $1:1$ and has a conical singularity.

We now apply  Kobyashi's result \cite{Koba2}, namely
if  $\tilde{S}$ be a complete maximal surface in $\mathbb{L}^{3}$ with at least one
conelike singularity such  that the Gauss map of $\tilde{S}$ is 1: 1, then $\tilde{S}$ is congruent
to the surface  defined by $\sqrt{x^{2}+y^{2}}+a\sinh(z/a)=0$, where $a$ is a nonzero
real constant.

Applying $\Phi$ to this we have that our timelike minimal surface is given by the transformation 
$y \to -iy$, and restriction to the reals. 

This yields that our timelike minimal surface is of the form $\sqrt{x^2-y^2}+a\textrm{sinh}(\frac{z}{a})=0$, where $a$ is a nonzero real constant. 
\end{proof}

Note that there is another transformation which maps timelike minimal surface to maximal surfaces and vice versa.
$(x,y,z) \to (ix,iy,iz)$ but making this transformation to $\sqrt{x^{2}+y^{2}}+a\sinh(z/a)=0$, where $a$ is a nonzero
real constant, does not preserve the property of  Gauss map being  $1:1$. In fact it gives the example which Kobyashi  mentions in his paper in a remark in article $1$.

\section{Euler-Ramanujan identity and timelike minimal  surfaces}

In Dey \cite{D}, and Dey and Singh \cite{DS}, various new identities were obtained from Weierstrass-Enneper representations of minimal surfaces in $\mathbb{R}^3$ and maximal surfaces in $\mathbb{L}^3$.\\
In this section, we obtain new identities from existing Euler-Ramanujan identities and Weierstrass-type representations of timelike minimal surfaces in $\mathbb{L}^3$. It is worth noting that there exist many  such representations of timelike minimal surfaces, the more notable of which are due to Magid\cite{Mag}, Lee\cite{Lee} and Kim et al\cite{sms}. The derivations of these representation formulae make use of novel concepts. One method, for example, involves \emph{paracomplex numbers}, also known as Lorentz numbers.  
\\
The Weierstrass representation formula for timelike minimal surfaces in $\mathbb{L}^3$, as derived in \cite{Lee}, is given by
\begin{align*}
    x=\frac{1}{2}\int(1+q^2)f(u)du-(1+r^2)g(v)dv,\\
    y=\frac{-1}{2}\int(1-q^2)f(u)du+(1-r^2)g(v)dv,\\
    z=-\int qf(u)du+rg(v)dv,
\end{align*}
where $u,v$ are null coordinates and $(q,r)$ is the projected Gauss map of the surface. The null coordinates are given by $u=x+y$ and $v=-x+y$, where $(x,y)$ are Lorentz isothermal coordinates for the minimal surface.

\subsubsection{Identity corresponding to Lorentz helicoid with spacelike axis}

The Weierstrass data corresponding to the Lorentz helicoid with spacelike axis is 
\begin{equation*}
    q(u)=-e^u,\,\,f(u)=-e^{-u},\,\,r(v)=e^{-v},\,\,g(v)=-e^v.
\end{equation*}
The Lorentz helicoid in parametric form is given by
\begin{equation*}
    X(u,v)=-(\textrm{sinh}u+\textrm{sinh}v,\textrm{cosh}u+\textrm{cosh}v,u+v).
\end{equation*}
In terms of the isothermal parameters $(x,y)$, the Lorentz helicoid is given by 
\begin{equation*}
    X^{'}(x,y)=-(2\textrm{cosh}x\textrm{sinh}y,2\textrm{cosh}x\textrm{cosh}y,2y).
\end{equation*}
Hence it is clear that the Lorentz helicoid with spacelike axis is the graph of the function 
\begin{equation*}
    \frac{X}{Y}=\textrm{tanh}\left(\frac{Z}{2}\right).
\end{equation*}
Now we define a surface $\tilde{X}$ as follows 
\begin{equation*}
    \tilde{X}(z,\bar{z})=-(\textrm{sinh}z+\textrm{sinh}\bar{z},\textrm{cosh}z+\textrm{cosh}\bar{z},z+\bar{z}).
\end{equation*}
where $z=y+ix$ and $\bar{z}=y-ix$ and $x,y$ are isothermal coordinates for the surface $X$ as before. Expanding the above expression in terms of the isothermal parameters $x,y$ yields
\begin{equation*}
    \tilde{X}^{'}(x,y)=-(2\textrm{cos}x\textrm{sinh}y,2\textrm{cos}x\textrm{cosh}y,2y)
\end{equation*}
Let 
\begin{align*}
    X=2\textrm{cos}x\textrm{sinh}y=\textrm{sinh}z+\textrm{sinh}\bar{z},\\
    Y=2\textrm{cos}x\textrm{cosh}y=\textrm{cosh}z+\textrm{cosh}\bar{z},\\
    Z=2y=2\textrm{Re}z.
\end{align*}
Again it is clear that $\tilde{X}$ is also the graph of the function $\frac{X}{Y}=\textrm{tanh}(\frac{Z}{2})$, from which we get 
\begin{equation*}
    \frac{X}{Y}=\frac{\textrm{sinh}(\frac{Z}{2})}{\textrm{cosh}(\frac{Z}{2})}=\frac{i\textrm{sin}(\frac{iZ}{2})}{\textrm{cos}(\frac{iZ}{2})}=\frac{i\textrm{cos}(\frac{iZ}{2}+\frac{\pi}{2})}{\textrm{cos}(\frac{iZ}{2})}.
\end{equation*}
In the first Ramanujan identity let $A=\frac{iZ}{2}$ and $X=\frac{\pi}{2}$. Then, using the parametric form of the surface $\tilde{X}$, we get
\begin{align*}
    \frac{\textrm{sinh}z+\textrm{sinh}\bar{z}}{\textrm{cosh}z+\textrm{cosh}\bar{z}}=i\prod_{k=1}^\infty\left(1-\frac{\frac{\pi}{2}}{(k-\frac{1}{2})\pi-i\textrm{Re}z}\right)\left(1+\frac{\frac{\pi}{2}}{(k-\frac{1}{2})\pi+i\textrm{Re}z}\right).
\end{align*}
This can be rewritten as follows to arrive at the identity corresponding to the Lorentz helicoid with spacelike axis
\begin{align*}
    \frac{\textrm{sinh}z+\textrm{sinh}\bar{z}}{\textrm{cosh}z+\textrm{cosh}\bar{z}}=i\prod_{k=1}^\infty\left(\frac{(k-1)\pi-i\textrm{Re}z}{(k-\frac{1}{2})\pi-i\textrm{Re}z}\right)\left(\frac{k\pi+i\textrm{Re}z}{(k-\frac{1}{2})\pi+i\textrm{Re}z}\right).
\end{align*}

\subsubsection{Identity corresponding to Lorentz helicoid with timelike axis}
The Weierstrass data corresponding to this surface is 
\begin{equation*}
    q(u)=\frac{\textrm{sin}u}{-1+\textrm{cos}u},\,\,f(u)=-1+\textrm{cos}u,\,\,r(v)=\frac{\textrm{sin}v}{1+\textrm{cos}v},\,\,g(v)=-(1+\textrm{cos}v).
\end{equation*}
The parametric form of the Lorentz helicoid is given by 
\begin{equation*}
    X(u,v)=-(u+v,\textrm{sin}u+\textrm{sin}v,-\textrm{cos}u-\textrm{cos}v),
\end{equation*}
where $(u,v)$ are null coordinates. In terms of the isothermal parameters $x,y$, the Lorentz helicoid with timelike axis is given by 
\begin{equation*}
    X^{'}(x,y)=-(2y,2\textrm{cos}x\textrm{sin}y,-2\textrm{cos}x\textrm{cos}y).
\end{equation*}
Thus it follows that the Lorentz helicoid with timelike axis can be expressed as the graph of the function
\begin{equation*}
    \frac{Y}{Z}=-\textrm{tan}\left(\frac{X}{2}\right).
\end{equation*}
Now we difene a surface $\tilde{X}$ as follows
\begin{equation*}
    \tilde{X}(z,\bar{z})=-(z+\bar{z}, \textrm{sin}z+\textrm{sin}\bar{z},-\textrm{cos}z-\textrm{cos}\bar{z}).
\end{equation*}
The above expression is expanded in terms of the insothermal parameters $x,y$ (of the surface $X$) which yields
\begin{equation*}
    \tilde{X}^{'}(x,y)=-(2y,2\textrm{cosh}x\textrm{sin}y,-2\textrm{cosh}x\textrm{cos}y).
\end{equation*}
Let
\begin{align*}
    X=2y=z+\bar{z}=2\textrm{Re}z\\
    Y=2\textrm{cosh}x\textrm{sin}y=\textrm{sin}z+\textrm{sin}\bar{z}\\
    Z=-2\textrm{cosh}x\textrm{cos}y=-\textrm{cos}z-\textrm{cos}\bar{z}.
\end{align*}
Hence it follows that $\tilde{X}$ is also the graph of the function $\frac{Y}{Z}=-\textrm{tan}(\frac{X}{2})$. This can be rewritten as 
\begin{equation*}
    \frac{Y}{Z}=-\frac{\textrm{sin}(\frac{X}{2})}{\textrm{cos}(\frac{X}{2})}=-\frac{\textrm{cos}(\frac{X}{2}+\frac{\pi}{2})}{\textrm{cos}(\frac{X}{2})}.
\end{equation*}
In the first Ramanujan identity, let $A=\frac{X}{2}$ and $X=\frac{\pi}{2}$. Then, using the parametric form of the surface $\tilde{X}$, we obtain
\begin{align*}
    \frac{\textrm{sin}z+\textrm{sin}\bar{z}}{\textrm{cos}z+\textrm{cos}\bar{z}}=\prod_{k=1}^\infty\left(1-\frac{\frac{\pi}{2}}{(k-\frac{1}{2})\pi-\textrm{Re}z}\right)\left(1+\frac{\frac{\pi}{2}}{(k-\frac{1}{2})\pi+\textrm{Re}z}\right).
\end{align*}
This can be rewritten as follows to arrive at the identity corresponding to the Lorentz helicoid with timelike axis
\begin{align*}
    \frac{\textrm{sin}z+\textrm{sin}\bar{z}}{\textrm{cos}z+\textrm{cos}\bar{z}}=\prod_{k=1}^\infty \left(\frac{(k-1)\pi-\textrm{Re}z}{(k-\frac{1}{2})\pi-\textrm{Re}z}\right)\left(\frac{k\pi+\textrm{Re}z}{(k-\frac{1}{2})\pi+\textrm{Re}z}\right).
\end{align*}
\\
\section{Acknowledgements}

I would like to thank my mentor Professor Rukmini Dey, ICTS, for her constant support, guidance and feedback throughout the duration of this project. I would also like to thank Rahul Singh for helpful discussions and comments. Finally, I would like to thank the staff of ICTS for their hospitality and for providing me with an optimal working environment throughout the duration of my visiting student fellowship.

\end{document}